\documentclass[reqno,a4paper,12pt]{article}%
\usepackage{amsmath}
\usepackage{amsfonts}%
\usepackage{amssymb}%
\usepackage{graphicx}
\usepackage[a4paper,includeheadfoot,margin=2.54cm]{geometry}
\providecommand{\U}[1]{\protect\rule{.1in}{.1in}}
\newtheorem{theorem}{Theorem}

\newtheorem{lemma}[theorem]{Lemma}

\newenvironment{proof}[1][Proof]{\noindent\textbf{#1.} }{\ \rule{0.5em}{0.5em}}
\begin{document}

\title{On the wave equation with logarithmic damping: wellposedness, blow-up and
numerical analysis}
\author{Salim Messaoudi$^{1}$, Mohammad Kafini$^{2}$ and Mostafa Zahri$^{3}$\\$^{1,3}$\small Department of Mathematics, \\ \small University of Sharjah, Sharjah 27272 UAE\\$^{2}$\small Department of Mathematics, The Interdisciplinary Research \\ \small Center for Construction and Building Materials, \\\small King Fahd University of Petroleum and Minerals, \\\small Dhahran 31261, Saudi Arabia\\$^{1}$\small E-mail: smessaoudi@sharjah.ac.ae\\$^{2}$\small E-mail: mkafini@kfupm.edu.sa\\$^{3}$\small E-mail: mzahri@sharjah.ac.ae}
\date{}
\maketitle

\begin{abstract}
In this work, we are concern with the wave equation subjected to a nonlinear
feedback of logarithmic type and nonlinear polynomial source. To achieve a
comprehensive understanding of this novel damping mechanism, we start with
studying the wellposedness and the uniqueness of the problem. Then we
establish the blow-up result of the problem under suitable initial data with
negative initial energy and critical exponent of the source term. Finally, we
present a brief numerical study and examples that support our result.

\textbf{Keywords} \textbf{and phrases: }Wellposedness, Logarithmic feedback,
Blow-up, Nonlinear source.

\textbf{AMS Classification : }35B40, 35B44, 35L05, 35L15, 35L70.

\end{abstract}

\section{Introduction}

In this work, we are concerned with the following initial-boundary-value
problem with logarithmic damping and nonlinear source%
\begin{equation}
\left\{
\begin{tabular}
[c]{ll}%
$u_{tt}-\Delta u+u_{t}\ln\left(  1+u_{t}^{2}\right)  =\left\vert u\right\vert
^{p-2}u,$ & $\Omega\times\left(  0,\infty\right)  ,$\\
$u(x,t)=0$ & $\partial\Omega,$\\
$u(x,0)=u_{0}\left(  x\right)  ,$ \ $u_{t}(x,0)=u_{1}\left(  x\right)  ,$ &
$\Omega,$%
\end{tabular}
\right.  \ \label{1}%
\end{equation}
where $\Omega$ is a bounded domain in $\mathbb{R}^{N}$, with a smooth boundary
$\partial\Omega,u_{0},u_{1}$ are given initial data in suitable spaces and $p>2$ is a constant satisfying some conditions to be specified later. Such equations arise in nonlinear dissipative systems where the damping grows faster than linear but slower than any superlinear.

The study of nonlinear wave equations plays a fundamental role in
understanding many physical systems, including elasticity, fluid mechanics,
and wave propagation in nonlinear media. A central problem in this theory is
to determine whether solutions exist globally in time or develop singularities
(blow-up) in finite time. In particular, the presence of damping and source
terms leads to a rich interplay between dissipative and energy-generating
mechanisms, which strongly influences the qualitative behavior of solutions.

The interaction between damping and source terms has been extensively studied
over the past decades. A pioneering contribution is due to Levine
$\cite{1},\cite{2}$, who first analyzed this interaction for wave equations
with linear damping and nonlinear source, using the concavity method. He
showed that solutions with negative initial energy blow up in finite time.
This approach was later extended by Georgiev and Todorova $\cite{3}$, who
generalized the results to nonlinear damping and identified conditions on the
exponents governing global existence versus blow-up. Subsequent works further
refined these results using energy methods, potential well theory, and
differential inequalities, establishing sharp thresholds depending on the
balance between damping strength and source growth. See $\cite{4}$%
-$\cite{10},$ for more results.

In the other hand, logarithmic nonlinearities have attracted growing interest
due to their appearance in various physical models, such as quantum mechanics,
control theory, etc ... . Wave equations with logarithmic source terms exhibit
distinctive analytical difficulties because of their slow growth and
non-polynomial structure. Recent works have investigated blow-up phenomena for
such equations, often combining energy techniques with Nehari manifold
arguments or Lyapunov functionals to establish finite-time blow-up under
appropriate conditions. See $\cite{11}$-$\cite{14}.$

More recently, attention has shifted toward nonstandard damping mechanisms,
such as fractional and logarithmic damping, which arise in models with memory
effects and complex dissipative structures. In this regard, Donghao Li et al
$\cite{15}$ studied the equation%
\[
u_{tt}-\Delta u+u_{t}\ln\left(  1+u_{t}^{2}\right)  =0
\]
and obtained a polynomial decay of the energy by employing the multipliers
technique. In $\cite{16},\cite{17},$ the authors incorporated the logarithmic
damping to different systems such as a one-dimensional Timoshenko system or a
piezoelectric beams system with magnetic effects and obtained some decay results.

These types of damping are typically weaker than the classical ones and, thus,
provide only limited energy dissipation. As a result, the competition between
them (logarithmic dampings) and source terms becomes more delicate. For this
reason, it is desirable to investigate the competition between the nonlinear
source and a logarithmic damping of the form mentioned in $\left(
\ref{1}\right)  $.

Let's mention that, t\textbf{o the best of our knowledge}, there is only one study by
Donghao Li et al. $\cite{18}$ about the blow up of the wave equation in the
presence of a nonlinear source term and a logarithmic damping of the form
$sgn\left(  u_{t}\right)  \ln\left(  1+\left\vert u_{t}\right\vert \right)  $.
In that work, the authors established the global existence of solutions by
using of Faedo-Galerkin combined with the stable set method. 
They also obtained the exponential decay of energy for solutions with certain
initial data. Furthermore, the blow-up in finite time under some conditions on
initial data was given.

The aim of this work is to establish, in details, the local well-posedness of
problem $\left(  \ref{1}\right)  $, using the Faedo-Galerkin method  and then
show that the solutions blow up under certain conditions on the initial data
and the exponent $p$. In addition, to illustrate our findings, we present a
brief numerical study and examples.

This paper consists of four sections. After the introduction, Section 2 presents the well-posedness in details. In section 3, we establish the blow up of the solution and in Section 4, we numerically illustrate the blow-up results of the energy functional. 

\section{Wellposedness}

Before we start, we would like to note here that we will use the following two
inequalities from $\cite{15}$%
\begin{align}
c_{p}\left\vert s\right\vert ^{1+p}  & \leq\left\vert s\right\vert \ln\left(
1+\left\vert s\right\vert \right)  \leq c_{q}\left\vert s\right\vert
^{q},\text{ \ for \ }\left\vert s\right\vert \leq1,\forall p\geq1,\text{
}0<q\leq2,\label{2}\\
c_{k}\left\vert s\right\vert ^{k}  & \leq\left\vert s\right\vert \ln\left(
1+\left\vert s\right\vert \right)  \leq c_{\rho}\left\vert s\right\vert
^{1+\rho},\text{ \ for \ }\left\vert s\right\vert \geq1,\forall\rho>0,\text{
}0<k\leq1.\nonumber
\end{align}
To establish the well-posedness of problem $\left(  \ref{1}\right)$ , we
first consider the following auxiliary problem%
\begin{equation}%
\begin{tabular}
[c]{ll}%
$u_{tt}-\Delta u+u_{t}\ln\left(  1+u_{t}^{2}\right)  =f\left(  x,t\right)  ,$
& in $\Omega\times\left(  0,\infty\right)  ,$\\
$u(x,t)=0,$ & on $\partial\Omega,$\\
$u(x,0)=u_{0}\left(  x\right)  ,$ \ $u_{t}(x,0)=u_{1}\left(  x\right)  ,$ & in
$\Omega,$%
\end{tabular}
\label{PF}%
\end{equation}
where $u_{0}$, $u_{1}$ and $f$ are given functions. We have the following theorem.

\begin{theorem}
\emph{Let the initial data }$(u_{0},u_{1})\in(H^{2}(\Omega)\cap H_{0}%
^{1}(\Omega))\times H_{0}^{1}(\Omega)$\emph{\ and suppose that }$f\in
L^{2}(0,T;H_{0}^{1}(\Omega)).$\emph{\ Then, problem }$\left(  \ref{PF}\right)
$\emph{\ has a unique strong solution satisfying}%
\begin{align*}
u  & \in L^{\infty}(0,T;H^{2}(\Omega)\cap H_{0}^{1}(\Omega)),\\
u_{t}  & \in L^{\infty}(0,T;H_{0}^{1}(\Omega)),\\
u_{tt}  & \in L^{\infty}(0,T;L^{2}(\Omega)).
\end{align*}
\end{theorem}
\noindent\textbf{Proof.} We use the standard Faedo-Galerkin method. Let
$\{w_{j}\}_{j=1}^{\infty}$ be an orthogonal basis of $H_{0}^{1}(\Omega)$
satisfying%
\[
-\Delta w_{j}=\lambda_{j}w_{j},\qquad\text{ }j=1,2,3,...
\]
Define the finite-dimensional space%
\[
V_{m}=\text{span}\{w_{1},\dots,w_{m}\}
\]
and look for approximate solutions of the form%
\[
u^{m}(x,t)=\sum_{j=1}^{m}a_{j}^{m}(t)w_{j}(x),
\]
in $V_{m}$, for the following approximate problem:%
\begin{equation}
\left\{
\begin{tabular}
[c]{l}%
$\int_{\Omega}\left[  u_{tt}^{m}w+\nabla u^{m}\cdot\nabla w+u_{t}^{m}%
\ln(1+(u_{t}^{m})^{2})w\right]  dx=\int_{\Omega}fwdx,\text{ }\forall w\in
V_{m}$\\
$u^{m}(0)=u_{0}^{m},\ u_{t}^{m}(0)=u_{1}^{m},$%
\end{tabular}
\right. \label{3}%
\end{equation}
where
\begin{equation}
u_{0}^{m}\longrightarrow u_{0}\text{ in }H^{2}(\Omega)\cap H_{0}^{1}%
(\Omega)\text{ \ and \ }u_{1}^{m}\longrightarrow u_{1}\text{ in }H_{0}%
^{1}(\Omega)\text{,\ }\label{31}%
\end{equation}
as $m\longrightarrow\infty.$ This is a system of ordinary differential
equations for the unknown functions $h_{j}^{m}\left(  t\right)  .$ Classical
ODE theory guarantees the existence of $C^{2}$- functions%
\[
h_{j}^{m}:[0,t_{m})\longrightarrow\mathbb{R},\text{ \ }j=1,2,...,m,\text{
}0<t_{m}\leq T.
\]
By replacing $w$ by $u_{t}^{m}$ in $\left(  \ref{3}\right)  $, we obtain
\[
\frac{d}{dt}\frac{1}{2}\int_{\Omega}\left[  (u_{t}^{m})^{2}+|\nabla u^{m}%
|^{2}\right]  \,dx+\int_{\Omega}(u_{t}^{m})^{2}\ln(1+(u_{t}^{m})^{2}%
)\,dx=\int_{\Omega}f\,u_{t}^{m}\,dx.
\]
Integration over $(0,t)$ yields
\begin{align*}
& \frac{1}{2}\Vert u_{t}^{m}\Vert_{2}^{2}+\frac{1}{2}\Vert\nabla u^{m}%
\Vert_{2}^{2}+\int_{0}^{t}\int_{\Omega}(u_{t}^{m})^{2}\ln(1+(u_{t}^{m}%
)^{2})\,dx\,ds\\
& =\int_{0}^{t}\int_{\Omega}f\,u_{t}^{m}\,dx\,ds+\frac{1}{2}\Vert u_{1}%
^{m}\Vert_{2}^{2}+\frac{1}{2}\Vert\nabla u_{0}^{m}\Vert_{2}^{2}.
\end{align*}
Young's inequality and convergence $\left(\ref{31}\right)$ imply
\[
\frac{1}{2}\sup_{0\leq t\leq t_{m}}\left[  \Vert\nabla u^{m}\Vert_{2}%
^{2}+\Vert u_{t}^{m}\Vert_{2}^{2}\right]  +\int_{0}^{t_{m}}\int_{\Omega}%
(u_{t}^{m})^{2}\ln(1+(u_{t}^{m})^{2})\,dx\,dt
\]%
\[
\leq\frac{1}{2}\Vert u_{1}\Vert_{2}^{2}+\frac{1}{2}\Vert\nabla u_{0}\Vert
_{2}^{2}+\varepsilon\int_{0}^{t_{m}}\Vert u_{t}^{m}\Vert_{2}^{2}\,dt+\frac
{1}{4\varepsilon}\int_{0}^{t_{m}}\int_{\Omega}f^{2}\,dx\,dt
\]%
\[
\leq\frac{1}{2}\Vert u_{1}\Vert_{2}^{2}+\frac{1}{2}\Vert\nabla u_{0}\Vert
_{2}^{2}+\varepsilon T\sup_{0\leq t\leq t_{m}}\Vert u_{t}^{m}\Vert_{2}%
^{2}+\frac{1}{4\varepsilon}\int_{0}^{T}\int_{\Omega}f^{2}\,dx\,dt.
\]
We then choose $\varepsilon$ so small that $\varepsilon T=\frac{1}{4}$.
Consequently, we obtain for some $C_{0}>0,$
\begin{align}
& \sup_{0\leq t\leq t_{m}}\Vert u_{t}^{m}\Vert_{2}^{2}+\sup_{0\leq t\leq
t_{m}}\Vert\nabla u^{m}\Vert_{2}^{2}+\int_{0}^{t_{m}}\int_{\Omega}(u_{t}%
^{m})^{2}\ln(1+(u_{t}^{m})^{2})\,dx\,dt\nonumber\\
& \leq\Vert u_{1}\Vert_{2}^{2}+\Vert\nabla u_{0}\Vert_{2}^{2}+\int_{0}^{T}%
\int_{\Omega}f^{2}\,dx\,dt=:C_{0}.\label{4}%
\end{align}
Next, we substitute $w=-\Delta u_{t}^{m}$ in $\left(  \ref{3}\right)  $ to
get
\[
\frac{1}{2}\frac{d}{dt}\left(  \Vert\nabla u_{t}^{m}\Vert_{2}^{2}+\Vert\Delta
u^{m}\Vert_{2}^{2}\right)  -\int_{\Omega}\Delta u_{t}^{m}\,u_{t}^{m}%
\ln(1+(u_{t}^{m})^{2})\,dx=-\int_{\Omega}\Delta u_{t}^{m}\,f\,dx,
\]
which gives
\begin{align*}
& \frac{1}{2}\frac{d}{dt}\left(  \Vert\nabla u_{t}^{m}\Vert_{2}^{2}%
+\Vert\Delta u^{m}\Vert_{2}^{2}\right)  +\int_{\Omega}|\nabla u_{t}^{m}%
|^{2}\ln(1+(u_{t}^{m})^{2})\,dx+\int_{\Omega}\frac{2|\nabla u_{t}^{m}%
|^{2}(u_{t}^{m})^{2}}{1+(u_{t}^{m})^{2}}\,dx\\
& =\int_{\Omega}\nabla u_{t}^{m}\cdot\nabla f\,dx.
\end{align*}
Integrating over $(0,t)$, using the convergence $\left(  \ref{31}\right)  ,$
leads to
\[
\frac{1}{2}\Vert\nabla u_{t}^{m}\Vert_{2}^{2}+\frac{1}{2}\Vert\Delta
u^{m}\Vert_{2}^{2}+\int_{0}^{t}\int_{\Omega}|\nabla u_{t}^{m}|^{2}\ln
(1+(u_{t}^{m})^{2})\,dx\,ds+\int_{0}^{t}\int_{\Omega}\frac{2|\nabla u_{t}%
^{m}|^{2}(u_{t}^{m})^{2}}{1+(u_{t}^{m})^{2}}\,dx\,ds
\]%
\[
\leq\int_{0}^{t}\int_{\Omega}\nabla u_{t}^{m}\cdot\nabla f\,dx\,ds+\frac{1}%
{2}\Vert\nabla u_{1}\Vert_{2}^{2}+\frac{1}{2}\Vert\Delta u_{0}\Vert_{2}^{2}%
\]%
\[
\leq\frac{1}{2}\Vert\nabla u_{1}\Vert_{2}^{2}+\frac{1}{2}\Vert\Delta
u_{0}\Vert_{2}^{2}+\varepsilon\int_{0}^{t}\int_{\Omega}|\nabla u_{t}^{m}%
|^{2}\,dx\,ds+\frac{1}{4\varepsilon}\int_{0}^{t}\int_{\Omega}|\nabla
f|^{2}\,dx\,ds
\]%
\[
\leq\frac{1}{2}\Vert\nabla u_{1}\Vert_{2}^{2}+\frac{1}{2}\Vert\Delta
u_{0}\Vert_{2}^{2}+\varepsilon T\sup_{0\leq t\leq t_{m}}\Vert\nabla u_{t}%
^{m}\Vert_{2}^{2}+\frac{1}{4\varepsilon}\int_{0}^{T}\int_{\Omega}|\nabla
f|^{2}\,dx\,ds.
\]
By choosing $\varepsilon$ such that $\varepsilon T=\frac{1}{4}$, then,
similarly to $\left(  \ref{4}\right)  $, we get for some $C_{1}>0,$%
\begin{equation}
\sup_{0\leq t\leq t_{m}}\Vert\nabla u_{t}^{m}\Vert_{2}^{2}+\sup_{0\leq t\leq
t_{m}}\Vert\Delta u^{m}\Vert_{2}^{2}+\int_{0}^{t_{m}}|\nabla u_{t}^{m}|^{2}%
\ln(1+(u_{t}^{m})^{2})\,\,dt\leq C_{1}.\label{5}%
\end{equation}
We also substitute $w=u_{tt}^{m}$ in $\left(  \ref{3}\right)  $. Thus, we
have
\[
\int_{\Omega}|u_{tt}^{m}|^{2}\,dx=\int_{\Omega}u_{tt}^{m}\,\Delta
u^{m}\,dx-\int_{\Omega}u_{tt}^{m}\,u_{t}^{m}\ln(1+(u_{t}^{m})^{2}%
)\,dx+\int_{\Omega}f\,u_{tt}^{m}\,dx
\]%
\begin{align}
& \leq\frac{1}{4}\int_{\Omega}|u_{tt}^{m}|^{2}\,dx+\int_{\Omega}(\Delta
u^{m})^{2}\,dx+\frac{1}{4}\int_{\Omega}(u_{tt}^{m})^{2}\,dx\nonumber\\
& +\int_{\Omega}|f|^{2}\,dx+\int_{\Omega}|u_{t}^{m}|^{2}\ln^{2}(1+(u_{t}%
^{m})^{2})\,dx.\label{6}%
\end{align}
Now, we show that the last term of $\left(  \ref{6}\right)  $ is bounded.  For
this, we consider the following partition
\[
\Omega^{+}=\{x\in\Omega:\ |u_{t}^{m}(x,t)|\geq1\},\qquad\Omega^{-}%
=\{x\in\Omega:\ |u_{t}^{m}(x,t)|<1\}.
\]
On $\Omega^{-}$, we have
\begin{equation}
\int_{\Omega^{-}}|u_{t}^{m}|^{2}\ln^{2}(1+(u_{t}^{m})^{2})\,dx\leq
|\Omega|\,\ln^{2}2<+\infty.\label{7}%
\end{equation}
By the algebraic relation $\left(  \ref{2}\right)  $ with $s=u_{t}^{m}$ and
the choice of $\rho>0$ such that
\[
2(1+\rho)\leq\frac{2N}{N-2},\quad\text{if }N\geq3,\qquad\rho>0,\ \text{if
}N=1,2
\]
and using the Sobolev embedding, we obtain by $\left(\ref{5}\right)  $:
\begin{equation}
\int_{\Omega^{+}}|u_{t}^{m}|^{2}\ln^2(1+(u_{t}^{m})^{2})\,dx\leq C\int%
_{\Omega^{+}}|u_{t}^{m}|^{2(1+\rho)}\,dx\leq C\Vert\nabla u_{t}^{m}\Vert
_{2}^{\,\rho+1}<+\infty.\label{8}%
\end{equation}
Combining $\left(  \ref{7}\right)  $,$\left(  \ref{8}\right)  $, we arrive at%
\[
\int_{\Omega}|u_{t}^{m}|^{2}\ln^{2}(1+(u_{t}^{m})^{2})\,dx<\infty.
\]
Therefore, $\left(  \ref{6}\right)  $-$\left(  \ref{8}\right)  $ yield%
\begin{equation}
\Vert u_{tt}^{m}\Vert_{2}^{2}\leq C_{2}.\label{9}%
\end{equation}
Putting $\left(  \ref{4}\right)  $, $\left(  \ref{5}\right)  $, $\left(
\ref{9}\right)  $ together, we conclude that the sequence $(u^{m})$ of
approximate solutions is uniformly bounded in $m$ and $t$. Therefore, we can
extend $t_{m}$ to $T$. Moreover, we can extract a subsequence (still denoted
by $(u^{m})$) such that, for some $u$, we have
\begin{equation}
\left\{
\begin{array}
[c]{c}%
u^{m}\rightharpoonup u\text{ weekly * in }L^{\infty}(0,T;H^{2}(\Omega)\cap
H_{0}^{1}(\Omega))\\
u_{t}^{m}\rightharpoonup u\text{ weekly * in }L^{\infty}(0,T;H_{0}^{1}%
(\Omega))\text{ \ \ \ \ \ \ \ \ \ \ \ \ }\\
u_{tt}^{m}\rightharpoonup u\text{ weekly * in }L^{\infty}(0,T;L^{2}%
(\Omega))\text{ \ \ \ \ \ \ \ \ \ \ \ \ }\\
u^{m}\rightharpoonup u\text{ weekly in }L^{2}(0,T;H^{2}(\Omega)\cap H_{0}%
^{1}(\Omega))\text{ \ }\\
u_{t}^{m}\rightharpoonup u\text{ weekly in }L^{2}(0,T;H_{0}^{1}(\Omega))\text{
\ \ \ \ \ \ \ \ \ \ \ \ \ \ }\\
u_{tt}^{m}\rightharpoonup u\text{ weekly in }L^{2}(0,T;L^{2}(\Omega)).\text{
\ \ \ \ \ \ \ \ \ \ \ \ \ \ }%
\end{array}
\right. \label{10}%
\end{equation}
By Aubin--Lions's theorem, this latter convergence, together with the compact
embedding of $H_{0}^{1}(\Omega)\hookrightarrow L^{2}(\Omega)$, we can extract
a further subsequence (still denoted by $(u^{m})$) such that
\[
u_{t}^{m}\rightarrow u_{t}\quad\text{strongly in }L^{2}(\Omega\times(0,T)),
\]%
\[
u_{t}^{m}\rightarrow u_{t}\quad\text{a.e. in }\Omega\times(0,T).\text{
\ \ \ \ \ \ \ \ \ \ \ \ }%
\]
Using the continuity of the mappings $g(s)=s\ln(1+s^{2})$, we deduce that
\begin{equation}
u_{t}^{m}\ln(1+(u_{t}^{m})^{2})\rightarrow u_{t}\ln(1+u_{t}^{2})\quad
\text{a.e. in }\Omega\times(0,T).\label{11}%
\end{equation}

On the other hand, by $\left(  \ref{7}\right)  $ and $\left(  \ref{8}\right)
$, we have
\[
\int_{\Omega}\left(  u_{t}^{m}\ln(1+(u_{t}^{m})^{2})\right)  ^{2}%
\,dx\leq\widetilde{C},
\]
for some constant $\widetilde{C}>0$. Consequently,
\[
\int_{0}^{T}\int_{\Omega}\left(  u_{t}^{m}\ln(1+(u_{t}^{m})^{2})\right)
^{2}\,dx\,dt\leq\widetilde{C}\,T<+\infty.
\]
Therefore, the subsequence $\left(  u_{t}^{m}\ln(1+(u_{t}^{m})^{2})\right)  $
is bounded in $L^{2}(\Omega\times(0,T))$. Combining this result with $\left(
\ref{11}\right)  $ and lemma of Lions $\cite{19}$, we infer
\begin{equation}
u_{t}^{m}\ln(1+(u_{t}^{m})^{2})\rightharpoonup u_{t}\ln(1+u_{t}^{2}%
)\quad\text{weakly in }L^{2}(\Omega\times(0,T)).\label{12}%
\end{equation}
Let's integrate $\left(  \ref{3}\right)  $ over $(0,t)$, $t<T$, to get
\begin{align*}
& \int_{\Omega}u_{t}^{m}w-\int_{\Omega}u_{1}^{m}w+\int_{0}^{t}\int_{\Omega
}\nabla u^{m}\cdot\nabla w+\int_{0}^{t}\int_{\Omega}u_{t}^{m}\ln(1+(u_{t}%
^{m})^{2})w\\
& =\int_{0}^{t}\int_{\Omega}fw,\text{ \ \ \ }\forall w\in V_{m}.
\end{align*}
Keeping in mind that $(u^{m})$ is the subsequence used in $\left(
\ref{9}\right)  $, the convergence $\left(  \ref{10}\right)  $-$\left(
\ref{12}\right)  $ will allow us to pass to the limit, as $m\rightarrow
\infty,$ and arrive at
\begin{equation}
\int_{\Omega}u_{t}w-\int_{\Omega}u_{1}w+\int_{0}^{t}\int_{\Omega}\nabla
u^{m}\cdot\nabla w+\int_{0}^{t}\int_{\Omega}u_{t}\ln(1+u_{t}^{2})\,w=\int%
_{0}^{t}\int_{\Omega}fw,\text{ \ }\forall w\in V_{m}.\label{13}%
\end{equation}
As $(V_{m})$ is dense in $H_{0}^{1}(\Omega)$, the result $\left(
\ref{13}\right)  $ remains valid for all $w\in H_{0}^{1}(\Omega)$.

By differentiating $\left(  \ref{13}\right)  $ with respect to $t$, we obtain
\[
\int_{\Omega}u_{tt}w+\int_{\Omega}\nabla u\cdot\nabla w+\int_{\Omega}u_{t}%
\ln(1+u_{t}^{2})\,w=\int_{\Omega}fw,
\]
for all $w\in H_{0}^{1}(\Omega)$ and for a.e. $t\in(0,T)$. Then we conclude,
by use of Green's formula, that
\[
u_{tt}-\Delta u+u_{t}\ln(1+u_{t}^{2})=f\quad\text{in }L^{2}(\Omega
\times(0,T)),
\]
which solves the equation in $\left(  \ref{PF}\right)  $.

To show uniqueness and that the solution satisfies the initial data, we
proceed in a standard way as in $\cite{16}$. This completes the proof of
Theorem 1. $\
\endproof
$

\begin{lemma}
For any $u\in L^{\infty}(0,T;H^{2}(\Omega)\cap H_{0}^{1}(\Omega))$, we have
\[
|u|^{p-2}u\in L^{2}(0,T;H_{0}^{1}(\Omega)),
\]
provided that
\begin{equation}
p\leq\frac{N}{N-4}+1,\qquad\text{if }N>4.\label{14}%
\end{equation}

\end{lemma}

\noindent\textbf{Proof. }We estimate%

\[
\int_{\Omega}\left\vert \nabla\left(  |u|^{p-2}u\right)  \right\vert
^{2}=(p-1)\int_{\Omega}\left(  |u|^{p-2}|\nabla u|\right)  ^{2}%
\]%
\[
\leq\frac{p-1}{2}\int_{\Omega}|\nabla u|^{2}+\frac{p-2}{2}\int_{\Omega
}|u|^{2(p-2)}.
\]
By using the embedding
\[
H^{2}(\Omega)\hookrightarrow L^{q}(\Omega),\qquad q\leq\frac{2N}%
{N-4},\ \text{if }N>4,
\]
and $q\geq1$, if $N=1,2,3,4$, we conclude that
\[
\int_{\Omega}\left\vert \nabla\left(  |u|^{p-2}u\right)  \right\vert ^{2}%
\leq\frac{p-1}{2}\int_{\Omega}|\nabla u|^{2}+\Vert u\Vert_{H^{2}}%
^{\,2(p-2)}<+\infty.
\]
Therefore $|u|^{p-2}u\in L^{2}(0,T;H_{0}^{1}(\Omega))$. \ $\
\endproof
$

Now, we are ready to establish the local existence for problem $\left(
\ref{1}\right)  .$

\begin{theorem}
Let $(u_{0},u_{1})\in(H^{2}(\Omega)\cap H_{0}^{1}(\Omega))\times H_{0}%
^{1}(\Omega)$. Suppose that $p$ satisfies $\left(  \ref{14}\right)  $. Then
$\left(  \ref{1}\right)  $ has a unique local solution
\[
u\in L^{\infty}((0,T);H^{2}\cap H_{0}^{1}(\Omega)),\qquad u_{t}\in L^{\infty
}((0,T);H_{0}^{1}(\Omega)),
\]%
\[
u_{tt}\in L^{\infty}((0,T);L^{2}(\Omega)).
\]

\end{theorem}
\textbf{Proof. }Let's consider the following set
\[
X_{T}^{m}=\left\{  w\in L^{\infty}(0,T;H^{2}(\Omega)\cap H_{0}^{1}%
(\Omega))\;;\;w_{t}\in L^{\infty}(0,T;H_{0}^{1}(\Omega)),\;w_{tt}\in
L^{\infty}(0,T;L^{2}(\Omega)),\right.
\]%
\[
\left.  w(\cdot,0)=u_{0},\quad w_{t}(\cdot,0)=u_{1},\quad\sup_{0\leq t\leq
T}\left[  \Vert w_{t}(\cdot,t)\Vert_{H_{0}^{1}}^{2}+\Vert w_{t}(\cdot
,t)\Vert_{2}^{2}\right]  \leq M^{2}\right\}  .
\]
$X_{T}^{m}$ is nonempty for $M$ large enough. See trace theorem $\cite{20}.$

We equip $X_{T}^{m}$ by the complete metric given by
\[
d^{2}(w,\widetilde{w})=\sup_{0\leq t\leq T}\left\{  \Vert w_{t}(\cdot
,t)-\widetilde{w}_{t}(\cdot,t)\Vert_{H_{0}^{1}}+\Vert w(\cdot,t)-\widetilde{w}%
_{t}(\cdot,t)\Vert_{2}^{2}\right\}  .
\]
Now, we define the map $\varphi$ on $X_{T}^{m}$ as follows: for $w\in
X_{T}^{m}$, we solve the linear problem
\begin{equation}%
\begin{cases}
u_{tt}-\Delta u+u_{t}\ln(1+u_{t}^{2})=|w|^{p-2}w, & \text{in }\Omega
\times(0,T),\\
u=0, & \text{on }\partial\Omega,\ t\geq0,\\
u(\cdot,0)=u_{0},\qquad u_{t}(\cdot,0)=u_{1}. &
\end{cases}
\label{15}%
\end{equation}
It follows, by Theorem 1, that $\left(  \ref{15}\right)  $ has a unique
solution
\[
u\in L^{\infty}(0,T;H^{2}(\Omega)\cap H_{0}^{1}(\Omega)),\text{ \ }u_{t}\in
L^{\infty}(0,T;H_{0}^{1}(\Omega)),\text{ \ }u_{tt}\in L^{\infty}%
(0,T;L^{2}(\Omega)).
\]
In addition, we have
\begin{align*}
& \frac{1}{2}\int_{\Omega}\left(  u_{t}^{2}+|\nabla u|^{2}\right)
\,dx+\int_{0}^{t}\int_{\Omega}|u_{t}|^{2}\ln(1+u_{t}^{2})\,dx\,ds\\
& =\frac{1}{2}\int_{\Omega}\left(  u_{1}^{2}+|\nabla u_{0}|^{2}\right)
\,dx+\int_{0}^{t}\int_{\Omega}u_{t}|w|^{p-2}w\,dx\,ds
\end{align*}%
\[
\leq\frac{1}{2}\int_{\Omega}\left(  u_{1}^{2}+|\nabla u_{0}|^{2}\right)
\,dx+\frac{1}{2}T\sup_{0<t<T}\int_{\Omega}u_{t}^{2}\,dx+\frac{1}{2}%
T\sup_{0<t<T}\int_{\Omega}|w|^{2(p-1)}\,dx\,ds
\]%
\[
\leq\frac{1}{2}\int_{\Omega}\left(  u_{1}^{2}+|\nabla u_{0}|^{2}\right)
\,dx+\frac{1}{2}T\sup_{0<t<T}\int_{\Omega}u_{t}^{2}\,dx+\frac{1}{2}C_{e}%
T\sup_{0<t<T}\Vert\nabla u\Vert_{2}^{2(p-1)},
\]
where $C_{e}$ is the embedding constant.

Thus, we obtain, for some $C^{\ast}>0$,
\[
(1-T)\sup_{0<t<T}\int_{\Omega}u_{t}^{2}+(1-C_{e}T\,M^{2(p-2)})\sup_{0<t<T}%
\int_{\Omega}|\nabla u|^{2}\leq C^{\ast}\int_{\Omega}\left(  u_{1}^{2}+|\nabla
u_{0}|^{2}\right)  .
\]
Let us choose $M$ such that
\[
M^{2}>2C^{\ast}\int_{\Omega}\left(  u_{1}^{2}+|\nabla u_{0}|^{2}\right)  ,
\]
then pick $T>0$ so small that
\begin{equation}
1-T\geq\frac{1}{2}\qquad\text{and}\qquad1-C_{e}T\,M^{2(p-2)}\geq\frac{1}%
{2}.\label{16}%
\end{equation}
Hence, we get
\[
\sup_{0<t<T}\int_{\Omega}u_{t}^{2}+\sup_{0<t<T}\int_{\Omega}|\nabla u|^{2}%
\leq2C^{\ast}\int_{\Omega}\left(  u_{1}^{2}+|\nabla u_{0}|^{2}\right)  \leq
M^{2}.
\]
Therefore, $u\in X_{T}^{m}$.

Hence, with the above choice of $M$ and $T$, $\varphi$ maps $X_{T}^{m}$ to
$X_{T}^{m}$. Next, we show that, for $T$ even smaller if necessary, $\varphi$
is a contraction. To this end, let
\[
u=\varphi(w)\text{ \ \ and \ \ }\widetilde{u}=\varphi(\widetilde{w}%
),\qquad\text{\ for \ \ }w,\widetilde{w}\in X_{T}^{m}.
\]
It is very easy to see that $v=u-\widetilde{u}$ solves
\begin{equation}%
\begin{cases}
v_{tt}-\Delta v+u_{t}\ln(1+u_{t}^{2})-\widetilde{u}_{t}\ln(1+\widetilde{u}%
_{t}^{\,2})=|w|^{p-2}w-|\widetilde{w}|^{p-2}\widetilde{w},\\
v=0\quad\text{on }\partial\Omega,\\
v(\cdot,0)=v_{t}(\cdot,0)=0.
\end{cases}
\label{17}%
\end{equation}
We multiply equation in $\left(  \ref{17}\right)  $ by $v_{t}$ and integrate
over $\Omega\times(0,t)$, $t<T$, to get
\[
\frac{1}{2}\int_{\Omega}\left(  v_{t}^{2}+|\nabla v|^{2}\right)  +\int_{0}%
^{t}\int_{\Omega}\left[  u_{t}\ln(1+u_{t}^{2})-\widetilde{u}_{t}%
\ln(1+\widetilde{u}_{t}^{\,2})\right]  (u_{t}-\widetilde{u}_{t})
\]%
\[
=\int_{0}^{t}\int_{\Omega}\left(  |w|^{p-2}w-|\widetilde{w}|^{p-2}%
\widetilde{w}\,\right)  v_{t}.
\]
By using the monotonicity of the function $g(s)=s\ln(1+s^{2})$, we obtain
\[
\sup_{0<t<T}\left[  \Vert v(\cdot,t)\Vert_{H_{0}^{1}}^{2}+\Vert v_{t}%
(\cdot,t)\Vert_{2}^{2}\right]  \leq T\,\sup_{0<t<T}\int_{\Omega}v_{t}^{2}%
+\int_{0}^{T}\int_{\Omega}\left(  |w|^{p-2}w-|\widetilde{w}|^{p-2}%
\widetilde{w}\right)  ^{2}.
\]
Recalling that $1-T\geq\frac{1}{2}$ and by the intermediate value theorem, we
have
\begin{equation}
d^{2}(u,\widetilde{u})\leq\lambda\int_{0}^{T}\int_{\Omega}|\theta
w+(1-\theta)\widetilde{w}|^{2(p-2)}\,|w-\widetilde{w}|^{2}\,dx\,dt,\text{ for
some }\lambda>0,\label{18}%
\end{equation}
where $0<\theta<1$. By using the embedding $H_{0}^{1}(\Omega)\hookrightarrow
L^{q}$ for
\[
q\geq1,\text{ if }N=1,2\text{ \ and \ }1\leq q\leq\frac{2N}{N-2}%
,\qquad\text{if }N\geq3,
\]
we arrive at
\[
\int_{\Omega}|\theta w+(1-\theta)\widetilde{w}|^{2(p-2)}|w-\widetilde{w}%
|^{2}\leq\left(  \int_{\Omega}|\theta w+(1-\theta)\widetilde{w}|^{2\gamma
(p-2)}\right)  ^{\frac{1}{\gamma}}\left(  \int_{\Omega}|w-\widetilde{w}%
|^{\frac{2\gamma}{\gamma-1}}\right)  ^{\frac{\gamma-1}{\gamma}}.
\]
By choosing $\gamma=\frac{N}{2}$, we infer
\[
\int_{\Omega}|\theta w+(1-\theta)\widetilde{w}|^{2(p-2)}|w-\widetilde{w}%
|^{2}\leq\left(  \int_{\Omega}|\theta w+(1-\theta)\widetilde{w}|^{N(p-2)}%
\right)  ^{\frac{2}{N}}\left(  \int_{\Omega}|w-\widetilde{w}|^{\frac{2N}{N-2}%
}\right)  ^{\frac{N-2}{N}}%
\]%
\[
\leq C\left[  \Vert\nabla w\Vert_{2}^{N(p-2)}+\Vert\nabla\widetilde{w}%
\Vert_{2}^{N(p-2)}\right]  \Vert\nabla w-\nabla\widetilde{w}\Vert_{2}^{2}%
\]%
\[
\leq2C\,\lambda M^{\frac{N(p-2)}{2}}\,\Vert\nabla w-\nabla\widetilde{w}%
\Vert_{2}^{2}.
\]
Therefore, $\left(  \ref{18}\right)  $ becomes
\[
d^{2}(u,\widetilde{u})\leq2C\,\lambda\,M^{\frac{N(p-2)}{2}}\,T\,d^{2}%
(w,\widetilde{w}).
\]
By taking $T$ even smaller so $\left(  \ref{16}\right)  $ remains valid and,
further,
\[
2C\,\lambda\,M^{\frac{N(p-2)}{2}}\,T<1,
\]
the mapping $\varphi$ is a contraction.

The Banach fixed-point theorem, then, guarantees that there exists a unique
$u\in X_{T}^{m}$ such that $\varphi(u)=u$. This is obviously the solution of
problem $\left(  \ref{1}\right)  $, for $T$ small enough. $\
\endproof
$

\section{Blow up}

In this section, we show that, under certain conditions on the initial data
and the nonlinearity, the local solution blows up infinite time. Before
stating our main result for this sections, we give a lemma whose proof can be
found in $\cite{6}$.

\begin{lemma}
\noindent\emph{There exists a positive }$C>1$\emph{, depending on }$\Omega
$\emph{\ only, such that}
\[
\left\Vert u\right\Vert _{p}^{s}\leq C\left(  ||\nabla u||_{2}^{2}+\left\Vert
u\right\Vert _{p}^{p}\right)  ,
\]
\emph{for any }$u\in H_{0}^{1}\left(  \Omega\right)  $\emph{\ and }$2\leq
s\leq p.$
\end{lemma}

Then, we define the energy functional by%
\begin{equation}
E\left(  t\right)  :=\frac{1}{2}\left(  \left\Vert u_{t}\right\Vert _{2}%
^{2}+\left\Vert \nabla u\right\Vert _{2}^{2}\right)  -\frac{1}{p}\left\Vert
u\right\Vert _{p}^{p}\label{19}%
\end{equation}
and, consequently, we have the following lemma.

\begin{lemma}
\emph{Along the solution of }$\left(  \ref{1}\right)  $, \emph{we have}
\[
E^{\prime}\left(  t\right)  :=-\int_{\Omega}u_{t}^{2}\ln\left(  1+u_{t}%
^{2}\right)  dx\leq0.
\]

\end{lemma}

\noindent\textbf{Proof.} We multiply $\left(  \ref{1}\right)  $ by $u_{t}$ and
integrate over $\Omega,$ we reach the result. $\
\endproof
$

Now, we are in position to state and prove our main theorem.

\begin{theorem}
Suppose that $2<p<\frac{N}{N-4}+1,$ if $n>4$ and $p>2$ if $N\leq4$ and let the
initial data $\left(  u_{0},u_{1}\right)  \in\left[  H^{2}\left(
\Omega\right)  \cap H_{0}^{1}\left(  \Omega\right)  \right]  \times H_{0}%
^{1}\left(  \Omega\right)  $ and satisfy \ $E\left(  0\right)  <0.$ Then, the
corresponding solution blows up in finite time.
\end{theorem}

\noindent\textbf{Proof.} \qquad We let%

\[
H\left(  t\right)  =-E\left(  t\right)  \geq0,
\]
therefore, \
\[
H^{\prime}\left(  t\right)  =-E^{\prime}\left(  t\right)  =\int_{\Omega}%
u_{t}^{2}\ln\left(  1+u_{t}^{2}\right)  dx\geq0.
\]

Also, for some $\varepsilon>0$ to be specified later, we define%
\[
L\left(  t\right)  =\varepsilon\int_{\Omega}uu_{t}dx+H^{1-\alpha}\left(
t\right)  ,
\]
where
\begin{equation}
0<\alpha<\frac{p-2}{2p}<1.\label{alfa}%
\end{equation}
Direct differentiation of $L\left(  t\right)  $, using $\left(  \ref{1}%
\right)  , $ leads to
\begin{align*}
L^{\prime}\left(  t\right)   & =\left(  1-\alpha\right)  H^{-\alpha}\left(
t\right)  H^{\prime}\left(  t\right)  +\varepsilon\left\Vert u_{t}\right\Vert
_{2}^{2}+\varepsilon\int_{\Omega}uu_{tt}dx\\
& =\left(  1-\alpha\right)  H^{-\alpha}\left(  t\right)  H^{\prime}\left(
t\right)  +\varepsilon\left\Vert u_{t}\right\Vert _{2}^{2}-\varepsilon
\left\Vert \nabla u\right\Vert _{2}^{2}+\varepsilon\left\Vert u\right\Vert
_{p}^{p}\\
& -\varepsilon\int_{\Omega}uu_{t}\ln\left(  1+u_{t}^{2}\right)  dx.
\end{align*}
Recalling the definition of $H(t),$ we have for, $0<a<1,$%
\begin{align*}
L^{\prime}\left(  t\right)   & =\left(  1-\alpha\right)  H^{-\alpha}\left(
t\right)  H^{\prime}\left(  t\right)  +\varepsilon\frac{p\left(  1-a\right)
+2}{2}\left\Vert u_{t}\right\Vert _{2}^{2}+\varepsilon\frac{p\left(
1-a\right)  -2}{2}\left\Vert \nabla u\right\Vert _{2}^{2}\\
& +\varepsilon p\left(  1-a\right)  H\left(  t\right)  +\varepsilon
a\left\Vert u\right\Vert _{p}^{p}-\varepsilon\int_{\Omega}uu_{t}\ln\left(
1+u_{t}^{2}\right)  dx.
\end{align*}
To handle the last term, we define
\[
\Omega^{-}=\left\{  x\in\Omega|\left\vert u_{t}\right\vert <1\right\}  \text{
\ and \ }\Omega^{+}=\left\{  x\in\Omega|\left\vert u_{t}\right\vert
\geq1\right\}  ;
\]
hence, we get%
\begin{align*}
\int_{\Omega^{-}}uu_{t}\ln\left(  1+u_{t}^{2}\right)  dx  & \leq\left(
\int_{\Omega^{-}}u^{2}dx\right)  ^{\frac{1}{2}}\left(  \int_{\Omega^{-}%
}\left(  \left\vert u_{t}\right\vert \ln\left(  1+u_{t}^{2}\right)  \right)
^{2}dx\right)  ^{\frac{1}{2}}\\
& \leq\left\Vert u\right\Vert _{2}\left(  \int_{\Omega^{-}}u_{t}^{2}\left(
\ln2\right)  \ln\left(  1+u_{t}^{2}\right)  dx\right)  ^{\frac{1}{2}}\\
& \leq\left\Vert u\right\Vert _{2}\left(  \int_{\Omega^{-}}u_{t}^{2}\ln\left(
1+u_{t}^{2}\right)  dx\right)  ^{\frac{1}{2}}.
\end{align*}
Young's inequality, for any $\delta_{1}>0,$ implies%
\[
\int_{\Omega^{-}}uu_{t}\ln\left(  1+u_{t}^{2}\right)  dx\leq\delta
_{1}\left\Vert u\right\Vert _{2}^{2}+\delta_{1}^{-1}H^{\prime}\left(
t\right)  .
\]
This inequality remains valid even if $\delta_{1}$ is time dependent since the
integral is taken over the $x-$variable. Therefore, by taking $\delta_{1}$ so
that $\delta_{1}^{-1}=kH^{-\alpha}\left(  t\right)  $, for some $k>0 $, then
we have%
\[
\int_{\Omega^{-}}uu_{t}\ln\left(  1+u_{t}^{2}\right)  dx\leq\frac{1}%
{k}H^{\alpha}\left(  t\right)  \left\Vert u\right\Vert _{2}^{2}+kH^{-\alpha
}\left(  t\right)  H^{\prime}\left(  t\right)  .
\]
From the definition of $H\left(  t\right)  ,$ we infer that
\[
H^{\alpha}\left(  t\right)  \leq\left(  \frac{1}{p}\right)  ^{\alpha
}\left\Vert u\right\Vert _{p}^{\alpha p};
\]
thus,%
\[
H^{\alpha}\left(  t\right)  \left\Vert u\right\Vert _{2}^{2}\leq B\left\Vert
u\right\Vert _{p}^{2+\alpha p},\text{ \ }B>0.
\]
Since $s=2+\alpha p<p,$ lemma 4 implies that%
\[
H^{\alpha}\left(  t\right)  \left\Vert u\right\Vert _{2}^{2}\leq C\left[
\left\Vert u\right\Vert _{p}^{p}-H\left(  t\right)  -\left\Vert u_{t}%
\right\Vert _{2}^{2}\right]
\]
and, consequently,%
\begin{equation}
\int_{\Omega^{-}}uu_{t}\ln\left(  1+u_{t}^{2}\right)  dx\leq\frac{C}{k}\left[
\left\Vert u\right\Vert _{p}^{p}-H\left(  t\right)  -\left\Vert u_{t}%
\right\Vert _{2}^{2}\right]  +kH^{-\alpha}\left(  t\right)  H^{\prime}\left(
t\right)  .\label{24}%
\end{equation}
On the other hand, using H\"{o}lder's inequality and relation $\left(
\ref{2}\right)  $, we have%
\begin{align*}
& \int_{\Omega^{+}}uu_{t}\ln\left(  1+u_{t}^{2}\right)  dx\\
& \leq\left(  \int_{\Omega^{+}}\left\vert u\right\vert ^{p}dx\right)
^{\frac{1}{p}}\left(  \int_{\Omega^{+}}\left(  \left\vert u_{t}\right\vert
\ln\left(  1+u_{t}^{2}\right)  \right)  ^{\frac{p}{p-1}}dx\right)
^{\frac{p-1}{p}}\\
& \leq\left(  \int_{\Omega^{+}}\left\vert u\right\vert ^{p}dx\right)
^{\frac{1}{p}}\left(  \int_{\Omega^{+}}\left\vert u_{t}\right\vert
^{\frac{p-2}{p-1}}\left(  u_{t}^{2}\ln\left(  1+u_{t}^{2}\right)  \right)
^{\frac{1}{p-1}}\left(  \ln\left(  1+u_{t}^{2}\right)  \right)  dx\right)
^{\frac{p-1}{p}}\\
& \leq\left(  \int_{\Omega^{+}}\left\vert u\right\vert ^{p}dx\right)
^{\frac{1}{p}}\left(  \int_{\Omega^{+}}\left\vert u_{t}\right\vert
^{\frac{p-2}{p-1}}\left[  \left(  u_{t}^{2}\right)  ^{1+\rho}\right]
^{\frac{1}{p-1}}\left(  \ln\left(  1+u_{t}^{2}\right)  \right)  dx\right)
^{\frac{p-1}{p}}.
\end{align*}
The choice of $\rho=\frac{p}{2}-1>0$ and Young's inequality, for any
$\delta_{2}>0,$ yield%
\begin{align}
\int_{\Omega^{+}}uu_{t}\ln\left(  1+u_{t}^{2}\right)  dx  & \leq\left(
\int_{\Omega^{+}}\left\vert u\right\vert ^{p}dx\right)  ^{\frac{1}{p}}\left(
\int_{\Omega^{+}}u_{t}^{2}\ln\left(  1+u_{t}^{2}\right)  dx\right)
^{\frac{p-1}{p}}\nonumber\\
& \leq C\left\Vert u\right\Vert _{p}\left(  H^{\prime}\left(  t\right)
\right)  ^{\frac{p-1}{p}}\label{25}\\
& \leq\delta_{2}\left\Vert u\right\Vert _{p}^{p}+c\left(  \delta_{2}\right)
H^{\prime}\left(  t\right)  .\nonumber
\end{align}
Combining $\left(  \ref{24}\right)  $ and $\left(  \ref{25}\right)  ,$ we get
\begin{align*}
& \int_{\Omega}uu_{t}\ln\left(  1+u_{t}^{2}\right)  dx\\
& \leq\frac{C}{k}\left[  \left\Vert u\right\Vert _{p}^{p}-H\left(  t\right)
-\left\Vert u_{t}\right\Vert _{2}^{2}\right]  +kH^{-\alpha}\left(  t\right)
H^{\prime}\left(  t\right)  +\delta_{2}\left\Vert u\right\Vert _{p}%
^{p}+c\left(  \delta_{2}\right)  H^{\prime}\left(  t\right)  .
\end{align*}
Consequently, we have%
\begin{align}
L^{\prime}\left(  t\right)   & \geq\left[  \left(  1-\alpha\right)
-\varepsilon kH^{-\alpha}\left(  t\right)  -\varepsilon c\left(  \delta
_{2}\right)  \right]  H^{\prime}\left(  t\right)  +\varepsilon\left(
\frac{p\left(  1-a\right)  +2}{2}+\frac{C}{k}\right)  \left\Vert
u_{t}\right\Vert _{2}^{2}\nonumber\\
& +\varepsilon\left(  p\left(  1-a\right)  +\frac{C}{k}\right)  H\left(
t\right)  +\varepsilon\left(  \frac{p\left(  1-a\right)  -2}{2}\right)
\left\Vert \nabla u\right\Vert _{2}^{2}\label{26}\\
& +\varepsilon\left[  a-\frac{C}{k}-\delta_{2}\right]  \left\Vert u\right\Vert
_{p}^{p}.\nonumber
\end{align}
As $p>2,$ we pick $0<a<1-\frac{2}{p}<1$ so that
\[
\frac{p\left(  1-a\right)  -2}{2}>0.
\]
Also, for $k$ large enough and $\delta_{2}$ small, we can obtain%
\[
a-\frac{C}{k}-\delta_{2}>0.
\]
Finally, we choose $\varepsilon$ small so that
\[
\left(  1-\alpha\right)  -\varepsilon\left[  kH^{-\alpha}\left(  0\right)
-c\left(  \delta_{2}\right)  \right]  >0
\]
and%
\[
L\left(  0\right)  =\varepsilon\int_{\Omega}u_{0}u_{1}dx+H^{1-\alpha}\left(
0\right)  >0
\]
which implies that%
\[
\left(  1-\alpha\right)  -\varepsilon\left(  kH^{-\alpha}\left(  t\right)
-c\left(  \delta_{2}\right)  \right)  >0,\text{ \ \ }\forall t\geq0.
\]

\noindent Therefore, $\left(  \ref{26}\right)  $ becomes%
\begin{equation}
L^{\prime}\left(  t\right)  \geq\varepsilon\left[  H\left(  t\right)
+\left\Vert u_{t}\right\Vert _{2}^{2}+\left\Vert \nabla u\right\Vert _{2}%
^{2}+\left\Vert u\right\Vert _{p}^{p}\right] \label{27}%
\end{equation}
and, hence,%
\[
L(t)\geq L(0)>0,\text{ \ \ \ }\forall t\geq0.
\]
In the other hand,%
\[
\left\vert \int_{\Omega}uu_{t}(x,t)dx\right\vert \leq\left\Vert u\right\Vert
_{2}\left\Vert u_{t}\right\Vert _{2}\leq C\left\Vert u\right\Vert
_{p}\left\Vert u_{t}\right\Vert _{2},
\]
which implies
\[
\left\vert \int_{\Omega}uu_{t}(x,t)dx\right\vert ^{\frac{1}{1-\alpha}}\leq
C\left\Vert u\right\Vert _{p}^{\frac{1}{1-\alpha}}\left\Vert u_{t}\right\Vert
_{2}^{\frac{1}{1-\alpha}}%
\]
and Young's inequality yields
\[
\left\vert \int_{\Omega}uu_{t}(x,t)dx\right\vert ^{\frac{1}{1-\alpha}}\leq
C\left[  \left\Vert u\right\Vert _{p}^{\frac{\mu}{1-\alpha}}+\left\Vert
u_{t}\right\Vert _{2}^{\frac{\theta}{1-\alpha}}\right]  ,
\]
where $1/\mu+1/\theta=1.$ The choice of $\theta=2\left(  1-\alpha\right)  $
will make $\mu/\left(  1-\alpha\right)  =2/\left(  1-2\alpha\right)  \leq p$
by $\left(  \ref{alfa}\right)  $. Therefore,
\[
\left\vert \int_{\Omega}uu_{t}(x,t)dx\right\vert ^{\frac{1}{1-\alpha}}\leq
C\left[  \left\Vert u\right\Vert _{p}^{s}+\left\Vert u_{t}\right\Vert _{2}%
^{2}\right]  ,
\]
where $s=\mu/\left(  1-\alpha\right)  .$ Using Lemma 4, we arrive at
\[
\left\vert \int_{\Omega}uu_{t}(x,t)dx\right\vert ^{\frac{1}{1-\alpha}}\leq
C\left(  \left\Vert u_{t}\right\Vert _{2}^{2}+||\nabla u||_{2}^{2}+\left\Vert
u\right\Vert _{p}^{p}\right)  .
\]
Consequently, we obtain
\begin{align}
L^{1/(1-\alpha)}(t)  & =\left[  H^{\left(  1-\alpha\right)  }(t)+\varepsilon
\int_{\Omega}uu_{t}(x,t)dx\right]  ^{1/(1-\alpha)}\nonumber\\
& \leq2^{1/(1-\alpha)}\left[  H(t)+\left\vert \int_{\Omega}uu_{t}%
(x,t)dx\right\vert ^{1/(1-\alpha)}\right] \label{28}\\
& \leq C\left[  H(t)+||u_{t}||_{2}^{2}+||\nabla u||_{2}^{2}+\left\Vert
u\right\Vert _{p}^{p}\right]  .\nonumber
\end{align}
Hence, $\left(  \ref{27}\right)  $ and $\left(  \ref{28}\right)  $ imply that,
for some $\mu>0,$
\[
L^{\prime}(t)\geq\mu L^{1/(1-\alpha)}(t).
\]
Integration over $(0,t)$ yields
\[
L^{\alpha/(1-\alpha)}(t)\geq\frac{1}{L^{-\alpha/(1-\alpha)}(0)-\mu\alpha
t/(1-\alpha)}.
\]
Therefore, the solutions cannot be exist after the time
\[
T^{\ast}=\frac{1-\alpha}{\mu\alpha L^{-\alpha/(1-\alpha)}(0)}.
\]
This completes the proof.$\ \ \ \ \ \ \ \ \ \ \
\endproof
$

\section{Numerical Examples}
In light of the obtained results, we present in this section four numerical experiments designed to validate the theoretical findings of Theorem~6 and to illustrate blow-up phenomena for problem~(1). 

The evolution of the system is analyzed through the energy functional defined in~(21), which serves as a key indicator for distinguishing between stable and unstable dynamics. The equation~(1) is discretized using a second-order finite-difference scheme in both time and space over the space-time domain.  
\[
[0,1]\times [0,T_{b}] = [0,1]\times [0,1],
\]
where \(T_b\) denotes the critical time at which the blow-up behavior starts to develop. Similar numerical constructions can be found in \cite{21, 22}.  

The numerical experiments are organized to highlight the blow-up regimes, where the solution undergoes rapid growth and the energy functional increases sharply, indicating loss of regularity in finite time.

All computations are performed under stability conditions ensuring consistency and convergence of the discrete model. Particular attention is given to accurately capturing the transition near \(T_b\), where the qualitative behavior of the solution changes from dissipative to unstable.
\subsection{One-dimensional case (Test 1)}
In this test, we investigate the behavior of the solution to equation~(1) in the \textbf{pre–blow-up regime}, with particular emphasis on the evolution of the energy functional defined in~(21). 

In the model~(1), the nonlinear exponent \(p\) is chosen sufficiently large so as to ensure that the initial energy satisfies \(E(0)<0\). This choice is crucial, as it places the system within the framework of Theorem~6, where so called negative initial energy is associated with the occurrence of blow-up in finite time. Consequently, the numerical simulations are expected to capture both the initial decay phase and the subsequent transition toward instability.

The governing system is discretized on the space–time domain $[0,1]\times [0,T]$ using a second-order finite-difference approximation. To preserve the conservative structure of the continuous system, the \textit{Lax–Wendroff} finite volume method is implemented, ensuring second-order accuracy in both time and space.

The discretization parameters satisfy the CFL stability condition:
\[
\Delta t = 0.0025 < \Delta x = 0.001.
\]

The initial conditions are given by
\begin{equation}
u(x,0) = (1-x)x, 
\qquad 
u_t(x,0) = 0, 
\quad \text{for } x \in [0,1].
\label{InSol1}
\end{equation}

These smooth and symmetric initial data generate a regular solution during the early stage of evolution. Figure~\ref{wavefewT1d} illustrates the temporal evolution of $u(x,t)$ at selected spatial points, together with solution snapshots, showing a gradual decrease in amplitude.

As long as the solution remains in the \textbf{pre–blow-up phase}, the associated energy functional~(21) exhibits a temporarily decay, as shown in Figure~\ref{energy1dBUP} (left panel). 

However, due to the condition \(E(0)<0\), this decay is only transient. As the system evolves, a qualitative change in behavior occurs: the energy ceases to decay and begins to increase rapidly, signaling the onset of the \textbf{blow-up regime}. This transition reflects the dominance of the nonlinear effects introduced by the large exponent \(p\).

The blow-up behavior is clearly illustrated in Figure~\ref{energy1dBUP} (right panel), where the energy functional grows sharply in finite time, indicating instability and loss of regularity of the solution.

This comparison highlights the transition from a temporarily stable dissipative regime to an unstable blow-up regime within the same model.

\begin{figure}[htp]
\centering
\includegraphics[width=0.3\textwidth,height=0.3\textwidth]{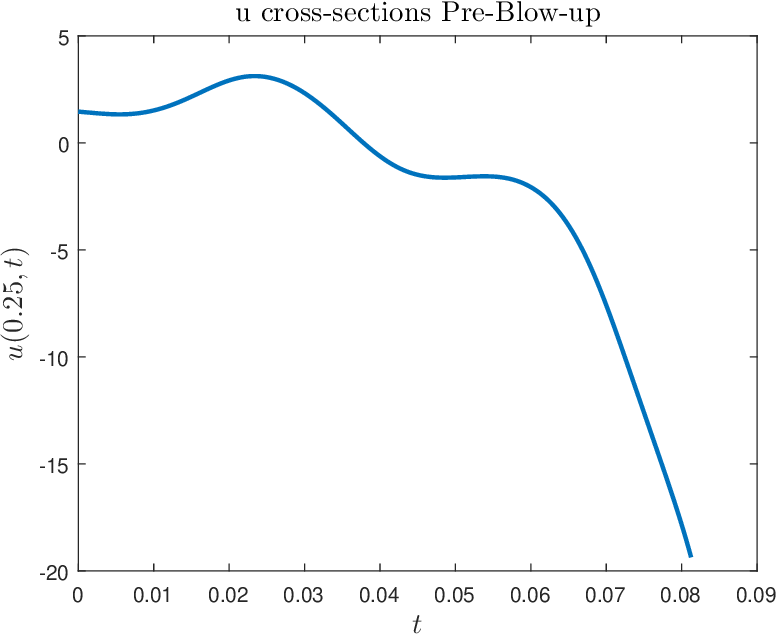}
\includegraphics[width=0.3\textwidth,height=0.3\textwidth]{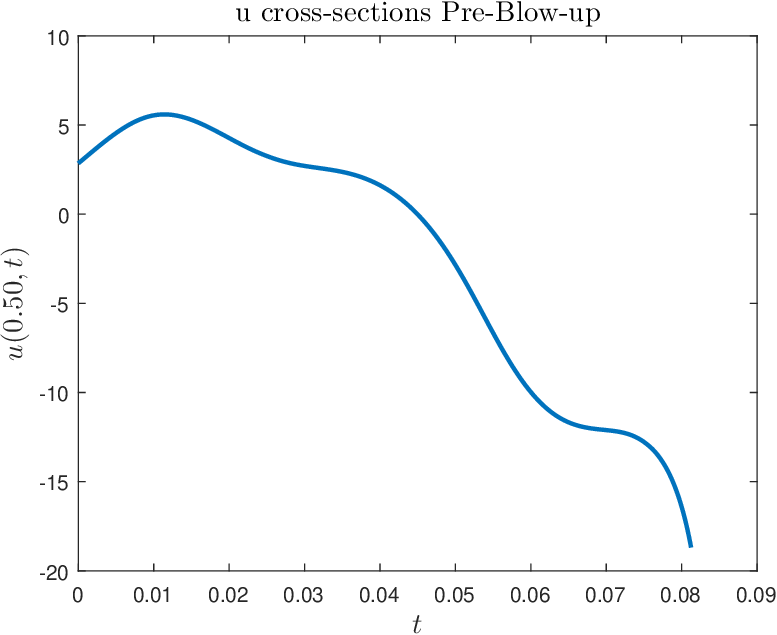}
\includegraphics[width=0.3\textwidth,height=0.3\textwidth]{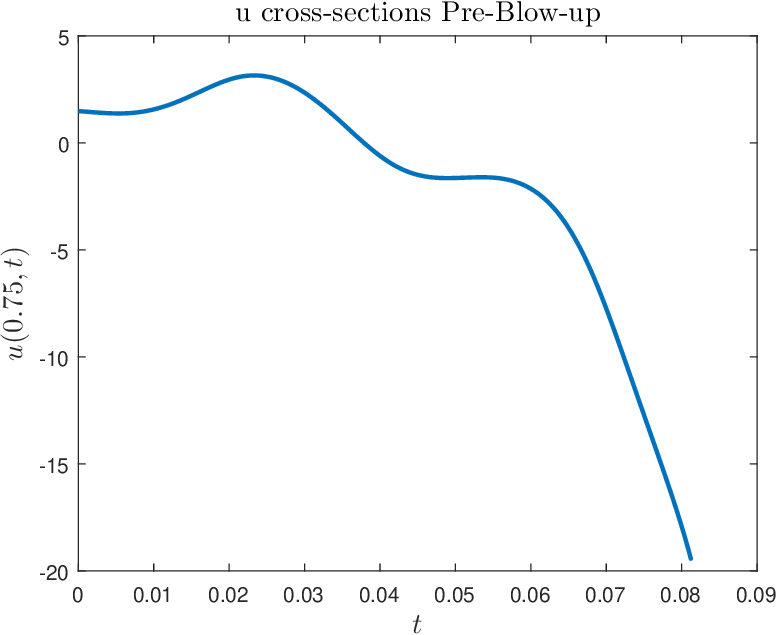}
\caption{TEST~1: Temporal evolution of $u(x,t)$ in the one-dimensional case.}
\label{wavefewT1d}
\end{figure}

\begin{figure}[htp]
\centering
\includegraphics[width=0.45\textwidth,height=0.35\textwidth]{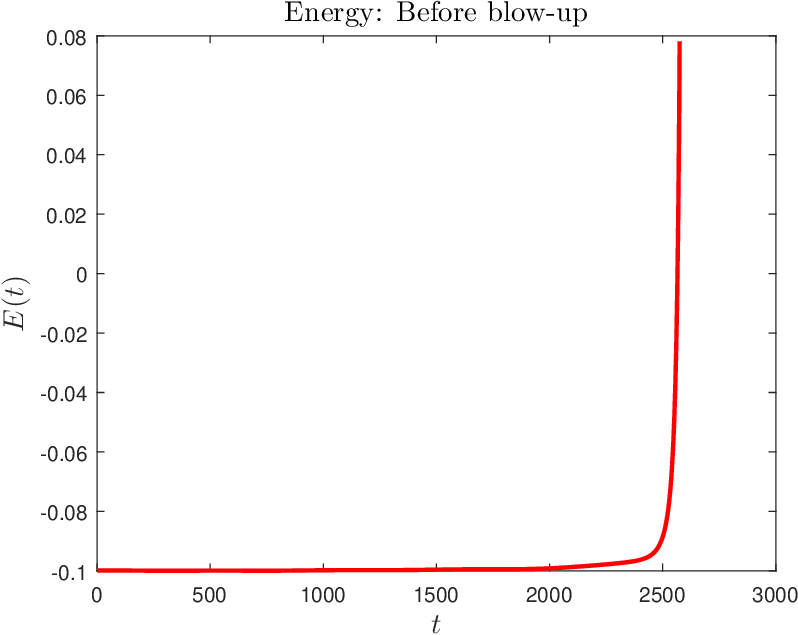}
\includegraphics[width=0.45\textwidth,height=0.35\textwidth]{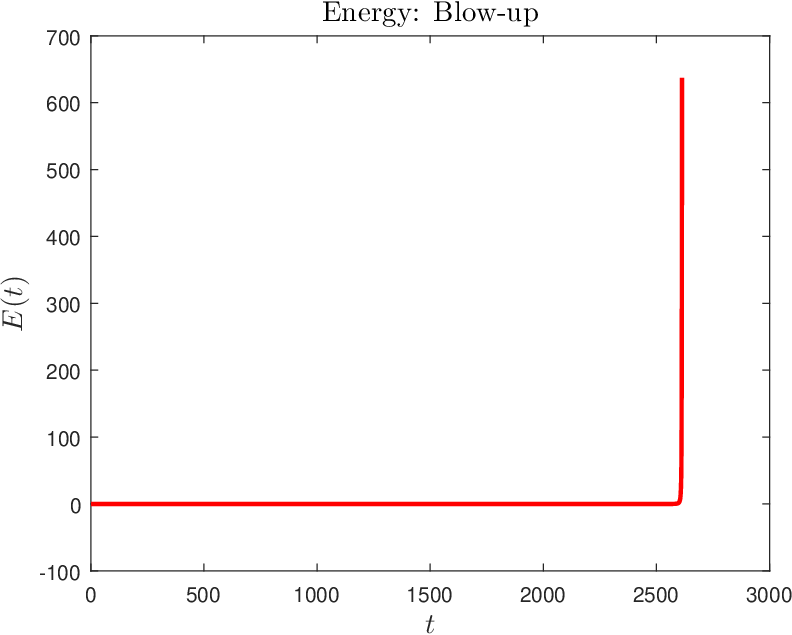}
\caption{TEST~1: The one dimensional Energy functional~(21) indicating the blow-up behavior.}
\label{energy1dBUP}
\end{figure}

The energy curve clearly exhibits a finite-time blow-up behavior: after an initial decay phase, the energy reaches a critical threshold and then increases rapidly, diverging in a short time interval. This sharp growth confirms the theoretical prediction associated with the condition \(E(0)<0\) and reflects the loss of regularity of the solution as the nonlinear effects become dominant.
\subsection{Two-dimensional case (Test 2)}

In the second numerical experiment, we investigate the two-dimensional behavior of equation~(1), with the aim of validating the predictions of Theorem~6 through the evolution of the energy functional defined in~(21). 

As in the one-dimensional case, the nonlinear exponent \(p\) is chosen sufficiently large to ensure that the initial energy satisfies \(E(0)<0\). This condition plays a fundamental role in the analysis, as it places the system in a regime where blow-up in finite time is expected. The numerical simulations are therefore designed to capture the transition toward instability.

The initial data are chosen as
\begin{eqnarray}\nonumber
u(x,y,0) &=& (1-x)(1-y)xy 
\left(200e^{-75\big((x-0.8)^2+(y-0.5)^2\big)} 
- e^{-75\big((x-0.5)^2+(y-0.8)^2\big)}\right),\\
\label{InSol2d}
u_t(x,y,0) &=& (1-x)(1-y)xy, 
\quad \text{for } (x,y)\in [0,1]\times[0,1].
\end{eqnarray}

The discretization uses $\Delta t = 0.0005$ and $\Delta x = \Delta y = 10^{-2}$, satisfying the CFL condition:
\begin{equation}\label{2dCFL}
\Delta t < \frac{(\Delta x)^2}{4}.
\end{equation}

During the \textbf{pre–blow-up regime}, the solution remains regular and the energy functional~(21) exhibits a temporarily decay phase. The amplitude decreases gradually while the spatial profile becomes smoother, as illustrated in Figure~\ref{wavefewT2d}. This slower decay, compared to the decay behavior observed in one dimension, is consistent with the influence of higher-dimensional effects on the dissipation mechanism.

However, due to the condition \(E(0)<0\), this decay is only transient. As time progresses, nonlinear effects become dominant and the system transitions to the \textbf{blow-up regime}. In this phase, the energy ceases to decrease and instead grows rapidly, indicating instability and loss of regularity in finite time.

This transition is clearly captured in Figure~\ref{energy2dBUP}:
\begin{itemize}
\item The left panel shows the temporarily decay in the pre–blow-up phase.
\item The right panel illustrates the blow-up behavior, where the energy increases sharply.
\end{itemize}

The sharp growth of the energy confirms the theoretical prediction of Theorem~6 and highlights the stronger instability mechanism present in the two-dimensional setting.

\begin{figure}[htp]
\centering
\includegraphics[width=0.45\textwidth,height=0.4\textwidth]{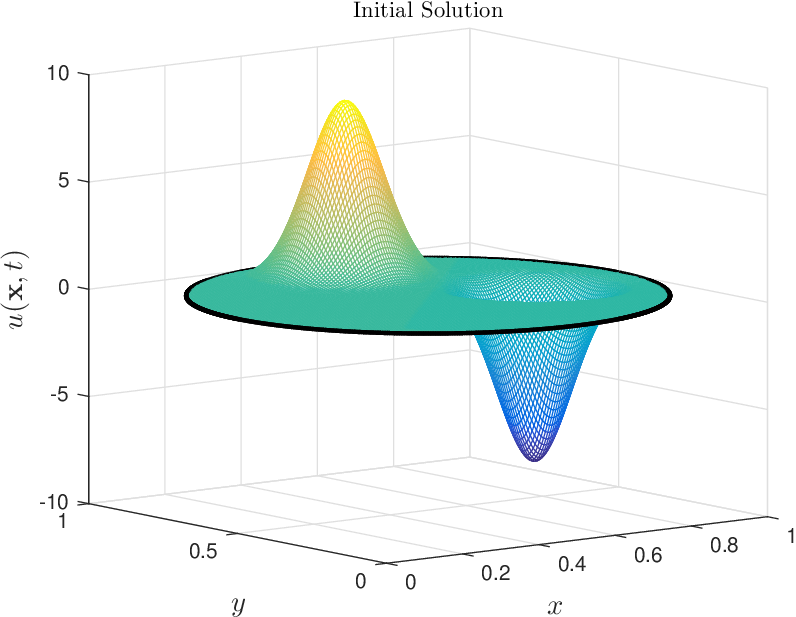}
\includegraphics[width=0.45\textwidth,height=0.4\textwidth]{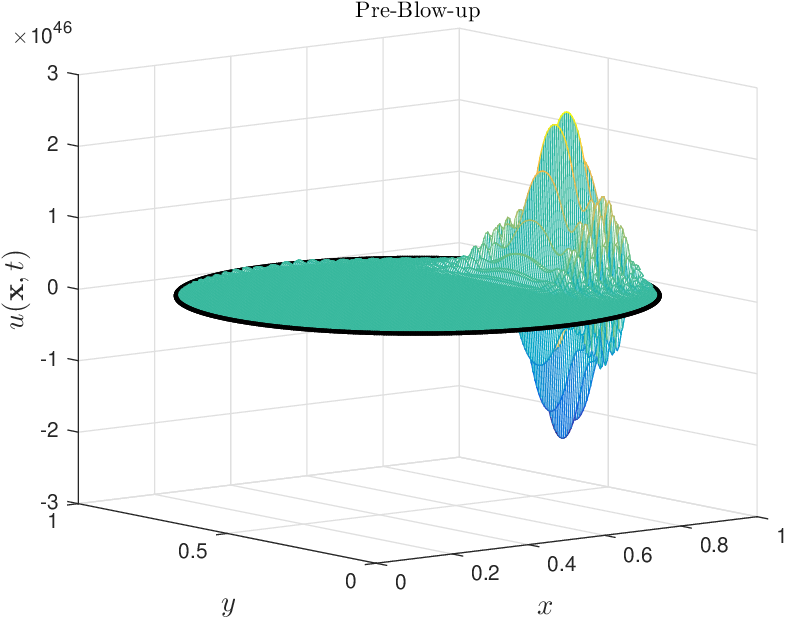}
\caption{TEST~2: Two-dimensional solution evolution: Initial Solution 
and an intermediary mesh of the function.}
\label{wavefewT2d}
\end{figure}

\begin{figure}[htp]
\centering
\includegraphics[width=0.45\textwidth,height=0.35\textwidth]{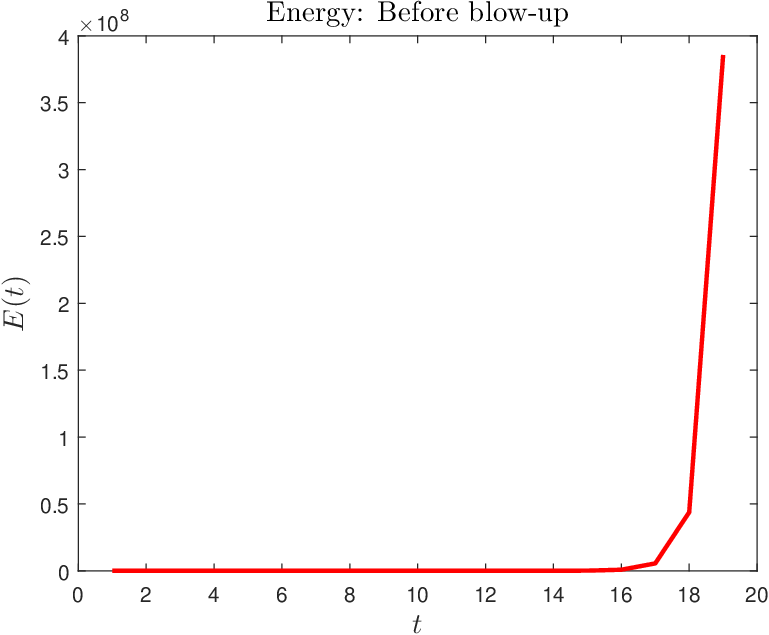}
\includegraphics[width=0.45\textwidth,height=0.35\textwidth]{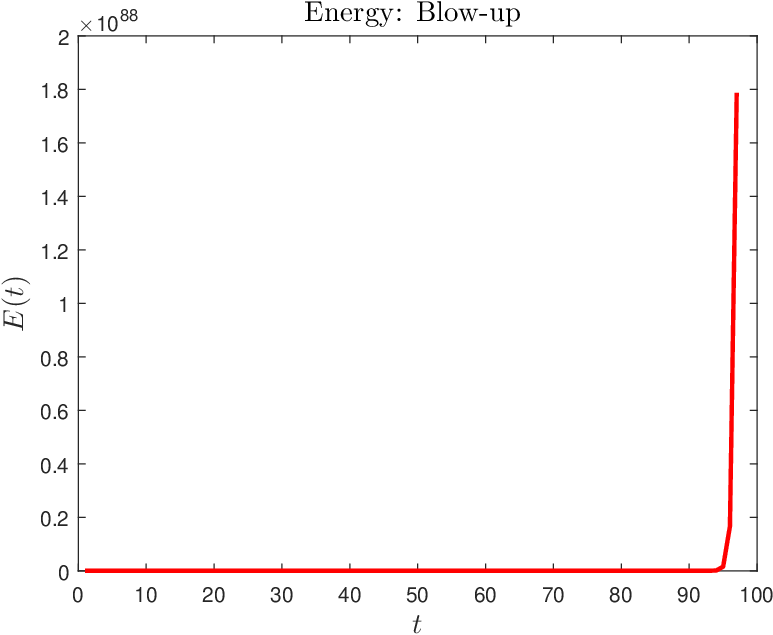}
\caption{TEST~2: The two dimensional Energy functional~(21) indicating the blow-up behavior.}
\label{energy2dBUP}
\end{figure}
$~$\\
\textbf{Remark}
The numerical results show a clear difference between the one-dimensional and two-dimensional cases. In one dimension, the solution passes through a transient regime in which the energy decays  before the blow-up mechanism becomes dominant. In two dimensions, the decay is slower, and the transition toward blow-up is more pronounced. This indicates that the instability mechanism is stronger in the two-dimensional setting, where the effect of the nonlinear terms overcomes the dissipative behavior more rapidly. In particular, the energy growth near the blow-up time appears steeper in two dimensions than in one dimension.

A possible explanation for the stronger blow-up behavior in two dimensions comes from the balance between dissipation and nonlinearity. In higher dimensions, the effect of the dissipative term is generally weaker than in one dimension. As a result, the damping mechanism becomes less effective in controlling the nonlinear source term. When the exponent $p$ is chosen large enough so that $E(0)<0$, the nonlinear effects dominate more strongly in two dimensions. This leads to a faster amplification of the energy and a more pronounced blow-up behavior. Compared with the one-dimensional case, the two-dimensional case exhibits a stronger blow-up mechanism. In particular, the pre--blow-up decay is slower in two dimensions, which means that dissipation is less effective in balancing the nonlinear source term. Consequently, once the condition $E(0)<0$ is satisfied, the energy grows more sharply and the onset of blow-up becomes more pronounced.

\begin{description}
\item[\textbf{Acknowledgment}] The first and the second author would like to
acknowledge the support provided by King Fahd University of Petroleum \&
Minerals (KFUPM), Saudi Arabia. While the third author acknowledges the
support provided by University of Sharjah. This work is partially sponsored by
the Interdisciplinary Research Center for Construction \& Building Materials (IRC-CBM).
\end{description}
\begin{description}
\item[\textbf{Availability of data and materials.}] No data were used to support this study.
\end{description}
\begin{description}
\item[\textbf{Competing interests.}] The authors declare no competing interests.
\end{description}

\end{document}